\pdfoutput=1
\documentclass[12pt]{article}

\usepackage{amssymb,amsmath,color}
\usepackage{graphicx}
\usepackage{subfigure}
\usepackage[arrow, matrix, curve,cmtip]{xy}

\usepackage{soul}

\usepackage{xy} \input xy \xyoption{all}

\usepackage{algpseudocode}
\usepackage{algorithm,float}

\usepackage{hyperref}

\newtheorem{lemma}{Lemma}
\newtheorem{theorem}[lemma]{Theorem}

\newtheorem{definition}[lemma]{Definition}
\newenvironment{proof}{\noindent{\em Proof:}}{\hfill $\Box$~\\}
\newtheorem{proposition}[lemma]{Proposition}
\newtheorem{remark}{Remark}
\newtheorem{example}[lemma]{Example}

\def\be{\begin{eqnarray}}
\def\ee{\end{eqnarray}}
\def\bes{\begin{eqnarray*}}
\def\ees{\end{eqnarray*}}

\def\R{\mathbb{R}}       
\def\C{\mathbb C}       
\def\cP{\mathcal P}
\def\cT{\mathcal{T}}
\def\cL{\mathcal{L}}

\def\vV{\mathcal V}
\def\sS{\mathcal S}

\def\cQ{\mathcal Q}
\def\cF{\mathcal F}

\def\toro{(\mathrm{HT})_{\om}}

\def\Q{\mathbb Q}
\def\N{\mathbb N}

\def\ot{\,\overline{t}\,}
\def\ox{\,\overline{x}\,}

\def\oy{\,\overline{y}\,}
\def\oz{\,\overline{z}\,}
\def\om{\,\overline{m}\,}

\def\dom{\mathrm{dom}}
\def\jac{\mathrm{Jac}}

\def\exp{\mathrm{e}}

\begin{document}

\title{Hybrid Trigonometric Varieties}

\author{Alberto Lastra, J. Rafael  Sendra \\
			 Dpto.\ de F\'isica y Matem\'aticas,
			 Universidad de Alcal\'a \\
             Ap. Correos 20,
E-28871 Alcal\'a de Henares,
Madrid (Spain) \\
{\tt alberto.lastra@uah.es} \\
{\tt rafael.sendra@uah.es} \\
Juana Sendra \\
Dpto.~Matem\'atica Aplicada a las TIC. \\ Universidad Polit\'ecnica de Madrid, Spain \\
{\tt juana.sendra@upm.es}
}
\date{}          
\maketitle

\begin{abstract}
In this paper we introduce the notion of hybrid trigonometric parametrization as a tuple of real rational expressions involving circular and hyperbolic trigonometric functions as well as monomials, with the restriction that variables in each block of functions are different. We analyze the main properties of the varieties defined by these parametrizations and we prove that they are exactly the class of real unirational varieties. In addition, we provide algorithms to implicitize and to convert a hybrid trigonometric parametrization into a unirational one, and viceversa.
\end{abstract}

\noindent {\bf  keywords:} Trigonometric parametrization, hyperbolic parametrization, implicitization algorithm,
unirational algebraic variety.


\section{Introduction}
In many problems, the parametric representation of geometric objects turns to be a fundamental tool. Clear examples of this claim may be found in some geometric constructions in computer aided design, like plotting, computing intersections, etc. (see \cite{HL}), or in computing integrals or solving differential equations (see e.g. \cite{fenggao2}, \cite{graseggerlastrasendrawinkler}). Probably the most common used parametrizations are the rational parametrizations (see \cite{HL} and \cite{libro}), but other types of parametrizations can also be applied, as radical parametrizations (see \cite{SSev1}, \cite{SSev2}) or trigonometric parametrizations (see \cite{hongschicho}, \cite{close}, \cite{pena}, \cite{sanchez}). Alternatively, one may work locally with parametrizations using rational or trigonometric splines  (see  \cite{uniform}, \cite{HL},\cite{close}, \cite{schicho-model}).

In this paper, we focus on the trigonometric-like type of parametrizations. The class of varieties studied in this paper is extended from the trigonometric curves considered in~\cite{hongschicho}, i.e. curves parametrized in terms of truncated Fourier series. This extension is made in three respects. On the one hand, we analyze varieties associated to the so-called hybrid trigonometric parametrizations in which not only circular trigonometric functions may appear, but also hyperbolic trigonometric and monomials; the idea of considering hyperbolic functions already appears in \cite{trigoHip}. On the other hand, not only polynomials are accounted, but also rational parametrizations of the previous form. Finally, the study is done for general real algebraic varieties and it is not restricted to the case of curves or surfaces.

So, we may be leading with a parametrization of the form (see Definition \ref{def:param-trig} for further details)
\[ \left(\dfrac{\sin(t_1)}{\cos(t_2)+\cosh(t_3)}, t_4+\sinh(t_3),t_4\cos(t_1),\dfrac{\cosh(t_3)}{t_4},\sin(t_2)\right). \]
We call this type of parametrizations hybrid trigonometric in the sense that the combine rationally elements from three different sets, namely
$$\{\sin(t_i),\cos(t_i)\}_{1\leq i \leq m_1}, \{\sinh(t_i),\cosh(t_i)\}_{m_1+1\leq i \leq m_2}, \{t_i\}_{m_2+1\leq i \leq m_3}.$$ Considering the parametrization as a real-valued function, and taking the Zariski closure of its image, we introduce the notion of hybrid trigonometric variety (see Definition \ref{definition-trigonometric}). We prove that hybrid trigonometric varieties are irreducible. Furthermore, we prove that they are precisely the (real) unirational varieties; we recall that real unirational means that it can be parametrized by means of real rational functions, but the corresponding function associated to the parametrization might not be injective. In addition, we provide algorithms to implicitize the hybrid trigonometric parametrization and to convert unirational parametrizations into hybrid trigonometric parametrization, and viceversa; for any prescribed triple $(m_1,m_2,m_3)$ such that $m_1+m_2+m_3$ is the dimension of the variety (see above the meaning of $m_i$).

At first glance one may notice no advance on this approach due to the absence of an enlargement of the class of unirational varieties when considering hybrid trigonometric parametrizations. However, a deepen study reveals that the appropriate point of view, considering either rational or trigonometric parametrizations, may lead to a more accurate solution of a problem under consideration; and hence being provided with conversion and implicitation algorithms enhances the applicability of the unirational varieties.

We devote a section to illustrate by examples some potential applications of hybrid trigonometric varieties. We comment some of them in this introduction.  Trigonometric curves emerge in numerous areas, as stated in~\cite{hongschicho}: classical curves, differential equations, Fourier analysis, etc. In addition, the important role of the varieties under study here are also shown up in recent studies. In the  work~\cite{naher} (see also ~\cite{chen}, ~\cite{kaplan}, and ~\cite{kolebaje} ) the extended generalized Riccati Equation Mapping Method is applied for the (1+1)-Dimensional Modified KdV Equation (see Subsection \ref{subsec-edo}). The authors arrive at different families of solutions, classified into soliton and soliton-like solutions (written in terms of rational functions of hyperbolic ones), and periodic solutions (written as rational functions of trigonometric ones), under different cases of the parameters involved. Applying the ideas in \cite{fenggao2}, and using the results in this paper, one may approach the problem transforming the trigonometric parametric parametrization induced by the solution into a rational one.

Another source of applications is the use of trigonometric functions to describe geometric constructions like offsets, conchoids, cissoids, epicycloids, hypocycloids, etc. In this case, one usually introduces polar parametrizations. A polar representation of a surface  is of the form
$$f(u,v)=\rho(u,v)\boldsymbol{k}(u,v),$$
where $\left\|\boldsymbol{k}(u,v)\right\|=1$ is a parametrization of the unit sphere, and $\rho(u,v)$ is a positive radius function; for references on this topic, we refer to~\cite{concoid1,concoid2}. This work is concerned with the case in which both $\rho$ and $\boldsymbol{k}$ are expressed as rational functions of $\{\cos(u),\sin(u),\cos(v),\sin(v)\}$. Indeed, Subsection \ref{subsec-epi-hypo} deals with the study of epicycloid and hypocycloid surfaces in which  $\boldsymbol{k}$ is chosen as the spherical coordinates in $\R^3$.

Trigonometric curves and surfaces provide a wide catalog of shapes  to be used in the application of the Hough transform to image processing. However, in order to apply the method one needs to deal with the implicit representation of the curves and surfaces in the catalog (see \cite{HT1}, \cite{HT2}). So, implicitization processes, as detailed in this paper, are required.

Other applications of this family of parametrizations are the interpolation of certain functions via quotients of trigonometric polynomials, as described in~\cite{hussain}, the plotting of curves and surfaces via trigonometric parametrizations (see Subsection \ref{subsec-plotting}) or the computation of intersection of geometric objects given in trigonometric form; in this last case, it is useful to transform the given parametrization into a rational one (see Subsection \ref{subsec-intersection}).

The paper is structured as follows. Section  \ref{sec-theory-def} is devoted to introduce the notions of hybrid trigonometric parametrization and variety and the first properties are studied. In Section \ref{sec:theory-properties} we analyze the fundamental properties of the hybrid trigonometric varieties and we see that they are characterized as the real unirational varieties. In Section \ref{sec-algorithms} we outline the algorithms derived from the proofs in the previous section, and in Section \ref{sec-aplicactions} we illustrate by examples the potential applicability of our results. The paper ends with a summary of conclusions.

Throughout this paper, we will be working with both the usual Euclidean topology and the Zariski topology. In case of ambiguity we will specify which topology is used.

\section{Hybrid Trigonometric Parametrizations and Varieties}\label{sec-theory-def}


We denote by $\dom(f)$ the domain of a function $f$, and  by $\jac(f)$ its Jacobian. Let $m,n\in \mathbb{N}$ such that $0<m<n$; $n$ will represent the dimension of the affine space where we work and $m$ the dimension of the variety. We denote $\ot=(t_1,\ldots,t_m)$. In addition, in the sequel we consider non-negative integers $m_1,m_2,m_3$ such that $m_1+m_2+m_3=m$. We will use the
notation $\om=(m_1,m_2,m_3)$.

In the next, we introduce the notion of $\om$--hybrid trigonometric parametrization that essentially  is an $n$-tuple depending on $m$ parameters such that $m_1$ of them appear as the variable of a sine or a cosine, $m_2$ appear as the variable of a hyperbolic sine or cosine and $m_3$ of them appear rationally. More precisely, we have the following the definition.

\begin{definition}\label{def:param-trig}
We say that an $n$-tuple
\[ \cT(\ot)=(T_1(\ot),\ldots,T_n(\ot)), \]
is an {\sf $\om$--hybrid trigonometric parametrization} if there exists a decomposition
$\{1,\ldots,m\}=J_1\cup J_2 \cup J_3$ with $\#(J_i)=m_i$, $i=1,2,3$ and there exist $\alpha_{ij},\alpha^{*}_{ij}\in \Q\setminus\{0\}$ and $\omega_{ij},\omega^{*}_{ij}\in \R$ such that
$$T_i(\ot)\in \R\left(\mathcal{A}_1,\mathcal{A}_2,\mathcal{A}_3\right)$$
where
\begin{equation*}\label{eq:aaa}
\left\{\begin{array}{lll} \mathcal{A}_1&=&\{\cos\left(\alpha_{i1} t_i +\omega_{i1}\right),\sin\left(\alpha_{i1}^* t_i +\omega_{i1}^*\right)\}_{i\in J_1}, \\  \mathcal{A}_2&=&\{\cosh\left(\alpha_{i2} t_i +\omega_{i2}\right),\sinh\left(\alpha_{i2}^* t_i +\omega_{i2}^*\right)\}_{i\in J_2}, \\
\mathcal{A}_3&=&\{t_j\}_{j\in J_3},
\end{array}
\right.
\end{equation*}
and  all parameters in $\ot$ do appear in $\cT(\ot)$.
Associated to $\cT$ we will consider the real function $\cT:\dom(\cT)\subset \R^m \rightarrow \R^n; \ot \mapsto \cT(\ot)$.
\end{definition}

\begin{remark}\label{re-0}
We will refer to $(m,0,0)$--parametrizations as {\sf circular trigonometric}, and  to $(0,m,0)$--parametrizations as {\sf hyperbolic trigonometric}. Note that $(0,0,m)$--parametrizations are precisely rational parametrizations. In Theorem \ref{theorem-TU}, we prove that three notions, namely hybrid trigonometric, circular trigonometric and hyperbolic trigonometric, are related.
\end{remark}

\begin{example}\label{ej-def}
The clearest examples of trigonometric parametrizations are the circle $(r\cos(t),r\sin(t))$ and the hyperbola $(r\cosh(t),r\sinh(t))$, with $r\in \R$, that are $(1,0,0)$ and $(0,1,0)$ parametrizations, respectively. Similarly  the spherical coordinates of an sphere generates a $(2,0,0)$-parametrization. All the previous examples are polynomial expressions of the trigonometric functions. However, in our case, we allow denominators. For instance,  $(1/\cos(t),\sin(t))$ is a circular trigonometric parametrization of the rational quartic $x^2y^2-x^2+1=0$. However, observe that $(t,\sin(t))$ is not a trigonometric parametrization (in the sense of our definition) since $t\not\in\R(\cos(t),\sin(t))$; note that, any rational expression of $\{\cos(t),\sin(t)\}$ is periodic. A similar case happens with the helix $(\cos(t),\sin(t),t)$. Finally, $(t_2\cos(t_1),t_2\sin(t_1),t_2)$ is a $(1,0,1)$--parametrization of the cone $x^2+y^2=z^2$.
\end{example}

\begin{remark} Note that in Definition \ref{def:param-trig} we have asked $\alpha_{ij}$ and $\alpha_{ij}^{*}$ to be rational numbers. The reason to exclude irrational numbers is that, in general, the Zariski closure of the image of the real function defined by the parametrization could be the whole affine space. For instance, if we take $\cT(t)=(\sin(t),\sin(\alpha t))$, with $\alpha\in \R\setminus \Q$, we prove that the Zariski closure of $\cT(\dom(\cT))$ is $\R^2$ and hence it does not define an algebraic curve as one may expect. More precisely, let $\beta\in \R$, and let
 $\Omega_\beta:=\{(\sin(\beta+2\pi n),\sin(\alpha \beta +\alpha 2 n \pi))\,|\,n\in \mathbb{N}\}\subset \{\cT(t)\,|\,t\in \R\}\cap \{(\sin(\beta),\lambda)\,|\,\lambda\in \R\}$. Now, we prove that $\Omega_\beta$ has infinitely many elements. Indeed, $\sin(\alpha\beta+2\alpha n\pi)$ turns out to be the imaginary part of $\exp^{i(\alpha\beta+2\alpha n\pi)}$. Two of these exponentials coincide for different $n_1,n_2\in\N$ if and only if
$$\alpha\beta+2\alpha n_1\pi= \alpha\beta+2\alpha n_2\pi+2n\pi,$$
for some $n\in\mathbb{Z}$. Then, it holds that $\alpha (n_1-n_2)= n$, which yields that $n_1=n_2$ and $n=0$ under the assumption that $\alpha\in\mathbb{R}\setminus{\Q}$. This implies that $\Omega_\beta$ must  be an infinite set. Therefore, by B\'ezouts' theorem, the Zariski closure of $\cT(\dom(\cT))$ contains all lines $x=\sin(\beta)$ with $\beta\in\R$. Thus, $\overline{\cT(\dom(\cT))}=\R^2$.
\end{remark}

\begin{lemma}\label{lemma-standard}
For every $\om$-hybrid trigonometric parametrization there exists a linear transformation $\cL:\R^m\rightarrow \R^m$ such that $\cT(\cL(\ot))$ is a tuple which entries belong to
$\R(\{\cos(t_i),\sin(t_i)\}_{i\in J_1},\{\sinh(t_i),\cosh(t_i)\}_{i\in J_2},\{t_i\}_{i\in J_3})$, where $J_1,J_2,J_3$ are as in Definition \ref{def:param-trig}. Furthermore, for every $i\in J_1$ and for every $j\in J_2$ the functions $\cos(t_i),\sin(t_i),\cosh(t_j),\sinh(t_j)$ appear in $\cT(\cL(\ot))$.
\end{lemma}
\begin{proof}
Let $\cT(\ot)$ be expressed as in Definition \ref{def:param-trig}. First we observe that, using the formulas of the sine, cosine (both circular and hyperbolic) of the addition of  angles, $\cT(\ot)$ can be expressed as a tuple with entries in $$\R(\{\cos(\alpha_{i,1} t_i),\sin(\alpha_{i1}^{*}t_i)\}_{i\in J_1},\{\sinh(\alpha_{i2} t_i),\cosh(\alpha_{i2}^{*}t_i)\}_{i\in J_2},\{t_i\}_{i\in J_3}).$$
Let us assume that $\cT(\ot)$ is already expressed as mentioned above.
Let $\ell_1$ be the  lcm of all denominators of $\{\alpha_{i1},\alpha_{i1}^{*}\}_{i\in J_1}$, and let $\ell_2$ be the  lcm of all denominators of $\{\alpha_{i2},\alpha_{i2}^{*}\}_{i\in J_2}$
Now, we consider the linear map defined as
\begin{equation}\label{eq:linearmap}
\cL_1(\ot)=\left\{ \begin{array}{ll} \ell_1 t_i & \text{if $i\in J_1$} \\
\ell_2 t_i & \text{if $i\in J_2$} \\
t_i & \text{if $i\in J_3$}
\end{array}
\right.
\end{equation}
Then, $\cT(\cL_1(\ot))$ is a tuple with entries in
$$\R(\{\cos(n_{i,1} t_i),\sin(n_{i1}^{*}t_i)\}_{i\in J_1},\{\sinh(n_{i2} t_i),\cosh(n_{i2}^{*}t_i)\}_{i\in J_2},\{t_i\}_{i\in J_3})$$
with $n_{i1},n_{i1}^{*},n_{i2},n_{i2}^{*}\in \mathbb{N}$. Finally, based on the expression of $\sin(kt)$, $\cos(kt)$, $\sinh(kt)$ and $\cosh(kt)$, with $k\in \mathbb{N}$, as polynomials expressions of $\sin(t),\cos(t),\sinh(t)$ and $\cosh(t)$, via Chebyshev polynomials (see 22.3.15 and the derivative of this formula; and 4.5.31 and 4.5.32 in~\cite{abram}) one gets that the entries of $\cT(\cL_1(\ot))$ belong to
$$\R(\{\cos(t_i),\sin(t_i)\}_{i\in J_1},\{\sinh(t_i),\cosh(t_i)\}_{i\in J_2},\{t_i\}_{i\in J_3}).$$
For the last part of the proof, let ${K_1}\subset J_1$ consist in those $i\in J_1$ for which either $\sin(t_i)$ or $\cos(t_i)$ does not appear in the $\cT(\cL_1(\ot))$; similarly, we introduce $K_2\subset J_2$. Then, we introduce the linear map
\begin{equation}\label{eq:linearmap2}
\cL_2(\ot)=\left\{ \begin{array}{ll} 2 t_i & \text{if $i\in K_1$} \\
2 t_i & \text{if $i\in K_2$} \\
t_i & \text{otherwise}
\end{array}
\right.
\end{equation}
 Using the expressions of $\cos(2t_i)$ and $\sin(2t_i)$ in terms of $\cos(t_i),\sin(t_i)$, and the corresponding expressions of $\cosh(2t_i)$ and $\sinh(2t_i)$ in terms of $\cosh(t_i),\sinh(t_i)$ we have that $\cT(\cL_2(\cL_1(\ot)))$ satisfies the required property.
\end{proof}

\begin{remark}\label{rem:pure}
Let $\cT(\ot)$ be a hybrid trigonometric parametrization as in Definition \ref{def:param-trig} and $\cT^*(\ot)=\cT(\cL(\ot))$ be the reparametrization in Lemma \ref{lemma-standard}. Then, taking into account that $\cL$ is a bijection we get that  $\cT(\dom(\cT))=\cT^*(\dom(\cT^*))$.
\end{remark}

\begin{example} In this example, we illustrate the statement in Lemma \ref{lemma-standard}; see also Algorithm \ref{alg:0}.
Let $n=4,m=3$, $\om=(1,1,1)$, and $a_1,a_2\in\mathbb{R}$. Let $\mathcal{T}$ be defined by
\begin{equation*}\label{e1}
\cT(\overline{t})=\left(\frac{\cos\left(a_1+\frac{1}{3}t_1\right)+t_3}{\sinh\left(\frac{1}{2}t_2\right)+t_3^2},\frac{\cos\left(\frac{1}{3}t_1\right)+t_3^2}{
\sinh\left(\frac{1}{2}t_2+a_2\right)+t_3},\frac{\cos\left(\frac{1}{3}t_1\right)}{\sinh\left(\frac{1}{2}t_2\right)+t_3},\frac{\cos\left(\frac{1}{3}t_1\right)+t_3}{
\sinh\left(\frac{1}{2}t_2\right)}\right),
\end{equation*}
that reads as follows
\begin{multline}
\mathcal{T}(\overline{t})=\left(\frac{A_1\cos\left(\frac{1}{3}t_1\right)-A_2\sin\left(\frac{1}{3}t_1\right)+t_3}{\sinh\left(\frac{1}{2}t_2\right)+t_3^2},
\frac{\cos\left(\frac{1}{3}t_1\right)+t_3^2}{A_4\cosh\left(\frac{1}{2}t_2\right)+A_3\sinh\left(\frac{1}{2}t_2\right)+t_3}\right.,\\
\left.\frac{\cos\left(\frac{1}{3}t_1\right)}{\sinh\left(\frac{1}{2}t_2\right)+t},\frac{\cos\left(\frac{1}{3}t_1\right)+t_3}{\sinh\left(\frac{1}{2}t_2\right)}\right),
\end{multline}
for $A_1=\cos(a_1)$, $A_2=\sin(a_1)$, $A_3=\cosh(a_2)$, $A_4=\sinh(a_2)$. The linear changes of parameters in (\ref{eq:linearmap}) and (\ref{eq:linearmap2}) are
$ \cL_1(\ot)=(3t_1,2t_2,t_3)$ and $\cL_2(\ot)=(t_1,t_2,t_3)$. So we get
\begin{multline}
\cT(\cL_2(\cL_1(\ot)))=
\left(\dfrac{A_1\cos(t_1)-A_2\sin(t_1)+t_3}{\sinh(t_2)+t_{3}^{2}}, \dfrac{\cos(t_1)+t_{3}^{2}}{\sinh(t_2) A_3+\cosh(t_2) A_4+t_3},\right. \\\left. \dfrac{\cos(t_1)}{\sinh(t_2)+t_3}, \dfrac{\cos(t_1)+t_3}{\sinh(t_2)}\right).
\end{multline}
\end{example}

\begin{definition}\label{def:pure}
We say that an $\om$--hybrid trigonometric parametrization $\cT(\ot)$ is {\sf pure} is it satisfies the properties in Lemma \ref{lemma-standard}, i.e. it is a tuple which entries belong to
$\R(\{\cos(t_i),\sin(t_i)\}_{i\in J_1},\{\sinh(t_i),\cosh(t_i)\}_{i\in J_2},\{t_i\}_{i\in J_3})$, where $J_1,J_2,J_3$ are as in Definition \ref{def:param-trig}, and for every $i\in J_1$ and for every $j\in J_2$ the functions $\cos(t_i),\sin(t_i),\cosh(t_j),\sinh(t_j)$ appear in $\cT(\ot)$.
\end{definition}

\noindent {\bf General assumptions:} In the sequel we will always assume that all parametrizations are pure. If this would not be the case, reparametrizing with $\cL_2\circ \cL_1$ (see (\ref{eq:linearmap}) and (\ref{eq:linearmap2})) the parametrization is transformed in pure form. In addition, for simplicity in the explanation, we assume w.l.o.g. that $J_1,J_2,J_3$ in Definition \ref{def:param-trig} are taken as $J_1=\{1,\ldots,m_1\}, J_2=\{m_{1}+1,\ldots,m_{1}+m_{2}\}, J_3=\{m_{1}+m_{2}+1,\ldots,m\}$, understanding that if  $m_i=0$ the corresponding $J_i$ is empty and the sets are swift to the left.

\begin{definition}\label{definition-trigonometric}
Let  $\vV\subset \R^n$ be a real algebraic variety of dimension $m$. We say that $\vV$ is a {\sf hybrid trigonometric variety} if there exists a hybrid trigonometric parametrizacion $\cT(\ot)$  such that $\vV$ is the Zariski closure of the image of $\cT$.
In this case, we say that $\cT(\ot)$ is a {\sf hybrid trigonometric parametrization} of $\vV$.
\end{definition}

\begin{remark}
In Definition \ref{definition-trigonometric}, if $\cT(\ot)$ is given as in Definition \ref{def:param-trig}, because of Remark \ref{rem:pure}, the pure parametrization
$\cT^*(\ot)$, provided by Lemma \ref{lemma-standard}, also satisfies the conditions imposed in Definition \ref{definition-trigonometric}.
\end{remark}

\section{Main Properties}\label{sec:theory-properties}

In this section, we use the notation as well as the hypotheses introduced in Section \ref{sec-theory-def}.
In addition,
in the sequel, we will consider  the hybrid $m$-dimensional torus
\begin{equation}\label{eq-toro}
\toro:=\sS^1\times \stackrel{m_1}{\cdots} \times \sS^1\times \mathcal{H}^{1}\times\stackrel{m_2}{\cdots}\times \mathcal{H}^{1}\times \R^{m_3}\subset \R^{2m_1+2m_2+m_3},
\end{equation}
where $\sS^1$ is the unit circle centered at the origin, and $\mathcal{H}^1$ stands for the hyperbola $\{(x,y):x^2-y^2=1\}$.
Observe that the implicit equations of $\toro$ are
\begin{equation}\label{eq:impl-toro} \left\{ x_{1}^{2}+x_{2}^{2}=1, \ldots,x_{2m_1-1}^2+x_{2m_1}^2=1, x_{2m_1+1}^{2}-x_{2m_1+2}^{2}=1, \ldots,x_{2m_1+2m_2-1}^2-x_{2m_1+2m_2}^2=1
\right\}.
\end{equation}
Furthermore, let
 \[ \xi(t)=\left(\dfrac{2 t}{t^2+1}, \dfrac{t^2-1}{t^2+1}\right),\qquad \nu(t)=\left(\dfrac{t^2+1}{2t}, \dfrac{t^2-1}{2t}\right)  \]
 be a proper parametrization of $\sS^1$ and $\mathcal{H}^1$, respectively. Then
\begin{equation}\label{eq-param-toro} \begin{array}{lcll}
\mathcal{M}: & \R^{m_1}\times (\R\setminus\{0\})^{m_2} \times \R^{m_3} & \longrightarrow & \toro \\
      & \ot         & \longmapsto     & \left(\xi(t_1),\ldots,\xi(t_{m_1}),\nu(t_{m_1+1}),\ldots,\nu(t_{m_1+m_2}),\right.\\
      &             &                 & \left. t_{m_1+m_2+1},\ldots,t_m \right).
\end{array} \end{equation}
is a proper parametrization of $\toro$ which inverse is
\begin{equation}\label{eq-inversa-toro}
\begin{array}{lcll}
\mathcal{M}^{-1} : & \toro \setminus F_{m_1}\subset \R^{2m_1+2m_2+m_3}
 &\rightarrow & \R^{m}  \\
 &  \ox & \mapsto & L(\ox)
 \end{array}
 \end{equation}
where
\begin{equation}\label{eq-L}
\begin{array}{lll}
L(\ox)&=&\left( \dfrac{x_1}{1-x_{2}},\ldots,\dfrac{x_{2m_1-1}}{1-x_{2m_1}},\dfrac{1}{x_{2m_1+1}-x_{2m_1+2}},\ldots,
\dfrac{1}{x_{2m_1+2m_2-1}-x_{2m_1+2m_2}},\right.
\\
&& \left. x_{2m_1+2m_2+1},\ldots,x_{2m_1+2m_2+m_3}\right),
\end{array}
\end{equation}
 and $F_{m_1}$ stands for the closed set
\begin{equation}\label{eq-dominioinversa}
 F_{m_1}:=\{(x_1,\ldots,x_{2m_1+2m_2+m_3})\in \toro\,|\,x_{2i}=1, \hbox{ for } i=1,\ldots,m_1 \}.
 \end{equation}

\begin{proposition}\label{prop-irred}
Every hybrid trigonometric variety is irreducible.
\end{proposition}
\begin{proof}
Let $\vV$ be  hybrid   of dimension $m$, and let $\cT(\ot)$, with $\ot=(t_1,\ldots,t_m)$, be an $\om$--hybrid trigonometric parametrization of $\vV$, which by our assumption is taken pure.
Let $\cF(\oy)$, where $\oy=(y_{1,1},y_{1,2},\ldots,y_{m_1,1},y_{m_1,2},\ldots,y_{m_1+m_2,1},y_{m_1+m_2,2},y_{m_1+m_2+1},\ldots,y_{m})$,  be the tuple obtained from $\cT(\ot)$ by replacing $\cos(t_i)$ (resp. $\sin(t_i)$) by $y_{i1}$ (resp. $y_{i2}$), with $1\leq i \leq m_1$ and $\cosh(t_j)$ (resp. $\sinh(t_j)$) by $y_{j1}$, (resp. $y_{j2}$), with $m_{1}+1\leq j\leq m_1+m_2$, and $t_k$ by $y_k$, with $m_1+m_2+1\leq k \leq m$.
We introduce the map
\begin{equation}\label{eq-psi}
 \begin{array}{lcll}
\Psi: & \dom(\cT) & \longrightarrow & \toro \subset \R^{2m_1+2 m_2+m_3} \\
      & \ot         & \longmapsto     & (\cos(t_1),\sin(t_1),\ldots,\cos(t_{m_1}),\sin(t_{m_1}),\\
			&\phantom{m}&\phantom{m}& \cosh(t_{m_1+1}),\sinh(t_{m_1+1}),\ldots,\cosh(t_{m_1+m_2}),\sinh(t_{m_1+m_2}), \\
            &           &           & t_{m_1+m_2+1}, \ldots, t_m).
\end{array}
\end{equation}
Let $H$ be the $\mathrm{lcm}$ of all denominators  in the tuple of rational function $\cF(\oy)$, and $\Omega=\toro\setminus\{\oy\,|\, H(\oy)=0\}$; observe that $\Omega\neq \emptyset$ since $\emptyset\neq \dom(\cT)$ because $\vV$  is its Zariski closure. So $\cF$
 induces the rational map
\begin{equation}\label{eq-cF}
 \begin{array}{lcll}
\cF: & \Omega\subset \toro & \dashrightarrow & \R^n \\
      & \oy         & \longmapsto     & \cF(\oy).
\end{array}
\end{equation}
Moreover, let $\mathcal{Z}$ be the Zariski closure $\mathcal{Z}=\overline{\cF(\Omega)}$. Since $\toro$ is irreducible, we have that $\mathcal{Z}$ is irreducible. Furthermore, $\dim(\mathcal{Z})\leq \dim(\toro)=m$ (see~\cite{kunz}, page 73, for a reference).
\begin{equation}\label{eq:diagram irreducible}
\xy
(0,0)*{\xy
	(15,10)*++{\mathcal{Z}}="ZZ";
    (24,10)*++{\subset \R^n}="Rn";
    (-20,-20)*++{\dom(\cT)}="Rm";
    (15,-20)*++{\toro}="Toro";
	{\ar^{\displaystyle \Psi} "Rm"; "Toro"};
	{\ar_{\displaystyle \cT} "Rm"; "ZZ"};
	{\ar@{-->}_{\displaystyle \cF} "Toro"; "ZZ"};
\endxy};
\endxy
\end{equation}
Since $\dom(\cT)=\Psi^{-1}(\dom(\cF))=\Psi^{-1}(\Omega)$ (see Diagram \ref{eq:diagram irreducible}), then $\cT(\ot)=\cF(\Psi(\ot))$ for $\ot\in \dom(\cT)$. Thus, $\vV =\overline{\cT(\dom(\cT))}\subset \overline{\cF(\Omega)}=\mathcal{Z}$.  Therefore, since $m=\dim(\vV)\leq\dim(\mathcal{Z})\leq m$, then $\dim(\mathcal{Z})=m$. Thus, since $\mathcal{Z}$ is irreducible and $\vV\subset \mathcal{Z}$, by~\cite{shaf} (see Theorem 1 in p. 68), we get that $\mathcal{Z}=\vV$, and hence $\vV$ is irreducible.
\end{proof}

\begin{proposition}\label{prop-UT} Every hybrid trigonometric variety is unirational over $\R$.
\end{proposition}
\begin{proof}
Let $\cT(\ot)$, with $\ot=(t_1,\ldots,t_m)$, be an $\om$--hybrid trigonometric parametrization of a hybrid trigonometric variety $\vV$, where $\dim(\vV)=m$.  Let $\cF(\oy), \Psi$ and $\mathcal{Z}$ as in the proof of Proposition \ref{prop-irred}, and  $\toro$ as in (\ref{eq-toro}).
Let $\mathcal{M}$ be as in (\ref{eq-param-toro}).
\begin{equation}\label{eq:diagram unirational}
\xy
(0,0)*{\xy
	(15,10)*++{\vV=\mathcal{Z}}="VV";
    (27,10)*++{\subset \R^n}="Rn";
    (-10,-20)*++{\dom(\cT)}="Rm";
    (15,-20)*++{\toro}="Toro";
    (80,-20)*++{\R^{m_1}\times (\R\setminus\{0\})^{m_2}\times \R^{m_3}}="Rm3";
	{\ar^{\displaystyle \Psi} "Rm"; "Toro"};
    {\ar_{\displaystyle \mathcal{M}} "Rm3"; "Toro"};
	{\ar_{\displaystyle \cT} "Rm"; "VV"};
	{\ar@{-->}_{\displaystyle \cF} "Toro"; "VV"};
	{\ar@{-->}_{\displaystyle  \mathcal{G}} "Rm3"; "VV"}
\endxy};
\endxy
\end{equation}
Then, $\mathcal{G}=\cF(\mathcal{M}(\ot))$ (see Diagram \ref{eq:diagram unirational}) is a real unirational parametrization with image in $\mathcal{Z}$. Since $\mathcal{M}$ is dominant in $\toro$ and $\mathcal{F}$ is dominant in $\mathcal{Z}$, then $\mathcal{G}$ is a real unirational parametrization of $\mathcal{Z}=\vV$.
\end{proof}

\begin{lemma}\label{lemma-inversa}
Let $p(\oz)\in \R[\oz]$, with $\oz=(z_1,\ldots,z_{m})$, be a non-zero polynomial and let $L(\ox)$, with $\ox=(x_1,\ldots,x_{2m_1+2m_2+m_3})$, be as in (\ref{eq-L}).  Then
\begin{enumerate}
\item  $p(L(\ox))$ is not identically zero.
\item Let $M(\ox)$ be the numerator of $p(L(\ox))$. It holds that $\toro \nsubseteq \{\overline{a} \in \R^{2m_1+2m_2+m_3}\,|\,M(\overline{a})=0\}$.
\end{enumerate}
\end{lemma}
\begin{proof} Let
\begin{multline}
\ox^*=(x_1,0,x_3,0,\ldots,x_{2m_1-1},0,\\
 \frac{1}{x_{2m_1+1}},0,\frac{1}{x_{2 m_1+3}},0,\ldots,\frac{1}{x_{2m_1+2m_2-1}},0,\\
x_{2m_1+2m_2+1},\ldots,x_{2m_1+2m_2+m}).
\end{multline}

Then, one has that $p(\ox)=p(L(\ox^*))$. It holds that, under the assumption that $p(L(x_1,\ldots,x_{2m}))=0$ then $p(x_1,x_3,\ldots,x_{2m-1})=0$. Hence, the first part of the statement follows.

Let us prove the second statement in the result. First we observe $M=p(L)N$ where $N$ is a polynomial of the form
$$(1-x_{2})^{\ell_1}\cdots(1-x_{2m_1})^{\ell_{m_1}}(x_{2m_1+1}-x_{2m_1+2})^{\ell_{m_1+1}}\cdots(x_{2m_1+2m_2-1}-x_{2m_1+2m_2})^{\ell_{m_1+m_2}},$$ for some $\ell_{i}\in \mathbb{N}\cup \{0\}$, $i=1,\ldots,m_1+m_2$. Substituting by $\mathcal{M}(\ot)$, see (\ref{eq-param-toro}), we get
$$ M(\mathcal{M}(\ot))=p(L((\mathcal{M}(\ot)))\,N(\mathcal{M}(\ot))=p(\ot)\,N(\mathcal{M}(\ot)).$$
Since $p$ is not zero, and $N(\mathcal{M}(\ot))$ is not zero either, then $M(\mathcal{M}(\ot))\neq0$ and hence then result follows.
\end{proof}

In the following lemma, let  $\toro$, $\Psi$ be as in (\ref{eq-toro}, \ref{eq-psi}).

\begin{lemma}\label{lemma-almost-surjectivity}
Let $\Theta$ be such that $\Psi(\Theta)$ is Zariski-dense in $\toro$, and  let $\Omega$ be
 a Zariski non--empty open subset of $\toro$. Then, $\Psi^{-1}(\Omega)\cap \Theta \neq \emptyset$.
\end{lemma}
\begin{proof}
Let $\Omega=\toro\setminus \Sigma$, with $\Sigma$ close. We first prove that $\Omega \cap \Psi(\Theta)\neq  \emptyset$. Indeed, let us assume that $\Omega \cap \Psi(\Theta)= \emptyset$. Then, $\Psi(\Theta)\subset \Sigma$.
Taking the Zariski closures we get $\toro=\Sigma$, which implies that $\Omega=\emptyset$, that is a contradiction. Now, let $\ox\in \Omega \cap \Psi(\Theta)$, then there exists $\ot\in \Theta$ such that $\Psi(\ot)=\ox$. So, $\ot\in  \Psi^{-1}(\Omega) \cap \Theta$.
\end{proof}

\begin{proposition}\label{prop-TU} Every unirational variety over $\R$ is hybrid trigonometric.
\end{proposition}
\begin{proof}
Let $\vV\subset \R^n$ be a unirational  variety over $\R$ with  $\dim(\vV)=m$. Fix a triple of non-negative integers  $\om=(m_1,m_2,m_3)$ such that $m_1+m_2+m_3=m$.
Let  $\toro \subset \R^{2m_1+2m_2+m_3}$, $\mathcal{M}$, $\mathcal{M}^{-1}$ and $\Psi$ as in (\ref{eq-toro},\ref{eq-param-toro},\ref{eq-inversa-toro},\ref{eq-psi}). Let   $$\cP(\ot)=\left(\dfrac{p_1(\ot)}{q_1(\ot)},\ldots,\dfrac{p_n(\ot)}{q_n(\ot)}\right), \,\,\text{with}\,\, \ot=(t_1,\ldots,t_m),$$
 be a rational real parametrization of $\vV$. We consider the map $\mathcal{Q}=\cP\circ \mathcal{M}^{-1} \circ \Psi$ (see Diagram \ref{eq:diagram 1}).
\begin{equation}\label{eq:diagram 1}
\xy
(0,0)*{\xy
	(-5,5)*++{\vV}="VV";
    (1,5)*++{\subset \R^n}="Rn";
    (-35,-20)*++{\R^{m_1}\times(\R\setminus\{0\})^{m_2}\times \R^{m_3}}="Rm";
    (35,-20)*++{\toro}="S";
    (55,-20)*++{\subset \R^{2m_1+2m_2+m_3}}="Rn2";
    (35,-45)*++{\R^{m}}="Rm2";
	{\ar^{\displaystyle \mathcal{M}} "Rm"; "S"};
	{\ar@{-->}_{\displaystyle \cP} "Rm"; "VV"};
	{\ar_{\displaystyle \Psi} "Rm2"; "S"};
	{\ar@{-->}_{\displaystyle \,\,\,\cP\circ \mathcal{M}^{-1}} "S"; "VV"}
\endxy};
\endxy
\end{equation}
Let us prove that $\mathcal{Q}(\ot)$ is an $\om$--hybrid trigonometric parametrization (see Definition \ref{def:param-trig}). For this purpose, we have to check that the formal
substitutions in $\mathcal{Q}$ are well defined. We first observe that by Lemma \ref{lemma-inversa} (1), the substitution $\cP \circ \mathcal{M}^{-1}(\ox)=\cP(L(\ox))$ is well defined. Now, let $M(\ox)$ be the numerator of  $q_i(L(\ox))$ for some $i$, and let us assume that $M(\Psi(\ot))$ is identically zero. This implies that  $\Psi(\R^m)$ is included
in the variety $W$ defined by $M$. Therefore, taking Zariski closures, and using that $\Psi$ is dominant in $\toro$, we get that $\toro=\overline{\Psi(\R^m)}\subset W$, which contradicts Lemma \ref{lemma-inversa} (2).

Now, we check that $\cQ$ satisfies the condition in Definition~\ref{definition-trigonometric}, namely that the Zariski closure of $\cQ(\dom(\cQ))$ is $\vV$. For this purpose,
we prove that  there exists an Euclidean open set $\emptyset\neq \Theta\subset \dom(\cQ)$,  such that the Zariski closure of $\cQ(\Theta)$ is $\vV$; from where one concludes the result since
\[ \vV=\overline{\cQ(\Theta)}\subset \overline{\cQ(\dom(\cQ))}\subset \overline{\vV}=\vV.
\]
\begin{itemize}
\item[(1)] We prove that $\cP\circ \mathcal{M}^{-1}$ is defined on a non-empty Euclidean open subset (in the induced topology) $\Omega_1$ of $\toro$.

We know that $\cP(\mathcal{M}^{-1}(\ox))$ is well-defined.
Let $H(\ox)$ be the $\mathrm{lcm}$ of the denominators of $\cP(\ot)$ and $g(\ox)$ the numerator of $H(L(\ox))$. We consider the close subsets of the $\toro$, $F_{m_1}$ (see~(\ref{eq-dominioinversa})) and $\Sigma:=\{\ox\in \toro \,|\, g(\ox)=0\}$. Clearly, the open subset $\toro\setminus F_{m_1}$ is not empty, and by Lemma~\ref{lemma-inversa} (2) it holds that $\toro\setminus\Sigma\neq\emptyset$. Moreover, since $\toro$ is irreducible, then $\Omega_1:=\toro\setminus\left(F_{m_1}\cup\Sigma\right)$ is a nonempty Euclidean open subset of the torus. Finally, let us see that $\Omega_1$ is included in the domain of $\cP\circ\mathcal{M}^{-1}$. Indeed, if $\ox\in \Omega_1$, then $\ox\notin F^{m}_{m_1}$ and hence $\mathcal{M}^{-1}(\ox)=L(\ox)$ is well-defined. Moreover, since $\ox\notin\Sigma$ then $g(\ox)\neq0$, and thus $H(L(\ox))\neq0$. So, $\cP(\mathcal{M}^{-1}(\ox))$ is well-defined.
 \item[(2)] We prove that $\mathcal{Q}$ is defined on a non-empty Euclidean open subset $\Theta$ of $\R^m$.  \\
 Since $\Psi$ is continuous and
$\Omega_1$ (see above) is open, then $\Theta:=\Psi^{-1}(\Omega_1)$ is Euclidean open in $\R^m$. Furthermore, using that $\Psi(\R^m)$ is Zariski dense in $\toro$, and that $\Omega_1$ is a non-empty Zariski subset of $\toro$, by Lemma \ref{lemma-almost-surjectivity} we get that $\Theta\neq \emptyset$.
\end{itemize}

We prove that $\cQ(\Theta)$   is Zariski dense in $\vV$. Let $f:=\cP\circ \mathcal{M}^{-1}$. Since $\Theta$ is Euclidean open in $\R^n$,
there exists open intervals $A_i$,$B_i$, $C_i$ such that
$\Theta^*:=\prod_{i=1}^{m_1} A_i \times \prod_{i=m_1+1}^{m_1+m_2} B_i \times \prod_{i=m_1+m_2+1}^{m} C_i\subset \Theta$. Therefore, $\Psi(\Theta^*)$ is the product of non-empty  arcs in the unit circle, in the
hyperbola, and segments in $\R$.  Thus, the Zariski closure of $\Psi(\Theta^*)$ is $\toro$. So, $\Psi(\Theta)$ is Zariski dense in $\toro$.

Since $f$ is continuous, we have that (see e.g. Theorem 7.2. pag. 44 in \cite{Willard}) $f(\,\overline{\Psi(\Theta)}\,)\subset \overline{f(\Psi(\Theta))}$, where all closures are w.r.t. the corresponding Zariski topologies. Now, since $\Psi(\Theta)$ is dense and $f$ is a dominant map, we finally get that
$$\vV=\overline{f(\,\overline{\Psi(\Theta)}\,)}\subset \overline{f(\Psi(\Theta))}=\overline{\cQ(\Theta)}\subset \vV.$$
Therefore, $\vV=\overline{\cQ(\Theta)}$.
\end{proof}

\begin{remark}\label{re-1} Observe that in the previous proposition, the tuple $\om$ is  freely chosen.
\end{remark}

\begin{remark} We observe that if in the proof of Proposition \ref{prop-TU} we consider a map $$\Psi^*=(\psi_1,\ldots,\psi_m):\R^m\rightarrow \R^m,$$ with
$\psi_i\in \R(\{\cos(t_i),\sin(t_i)\}_{i\in J_1},\{\sinh(t_i),\cosh(t_i)\}_{i\in J_2},\{t_i\}_{i\in J_3})$, such that $\cP(\Psi^*(\ot))$ is well-defined, and $\Psi^*(\dom(\Psi^*))$ is Zariski dense in $\R^m$,
then the trigonometric parametrization $\cQ(\ot)$ can be taken as $\cP(\Psi^*(\ot))$.
\end{remark}

We finish this section with the main theorem.

\begin{theorem}\label{theorem-TU} Let $\vV$ an irreducible variety. The following statements are equivalent
\begin{enumerate}
\item $\vV$ is unirational over $\R$.
\item $\vV$ is hybrid trigonometric.
\item $\vV$ is circular trigonometric.
\item $\vV$ is hyperbolic trigonometric.
\end{enumerate}
\end{theorem}
\begin{proof}
(1) implies (2),(3),(4) follows from Proposition \ref{prop-TU} taking $\om=(m_1,m_2,m_3)$, $\om=(m,0,0)$ and $\om=(0,m,0)$, respectively. (2), (3), (4) follows from Proposition \ref{prop-UT}.
\end{proof}


\section{Parametrization and Implicitization Algorithms}\label{sec-algorithms}

The proofs in the previous sections are constructive, and hence provide algorithms to deal with hybrid trigonometric varieties. In this section, we derive these
algorithms that, essentially, show how to change from hybrid trigonometric parametrizations to rational parametrizations and how to implicitize.

We start outlining the procedure to transform a trigonometric parametrization in pure form (see Definition \ref{def:pure}).

\begin{algorithm}[H]
	\caption{{\bf [ConvertPure]} Convert a hybrid trigonometric parametrization in pure form (see Def. \ref{def:pure}).} \vspace*{2mm}
	\label{alg:0}
	\begin{algorithmic}[1]

		\Require{An $\om$--hybrid trigonometric parametrization $\mathcal{T}(\overline{t})$ as in Def. \ref{def:param-trig}.}
		\Ensure{A pure $\om$--hybrid trigonometric reparametrization of $\cT(\ot)$.}
        \State Apply to $\cT(\ot)$ the formulas of the sine, cosine (both circular and hyperbolic) of the addition of  angles.
        \State  Compute $\cT(\cL_1(\ot))$, where $\cL_1(\ot)$ is as in (\ref{eq:linearmap}).
        \State  Apply to $\cT(\ot)$ the expression of $\sin(kt)$, $\cos(kt)$, $\sinh(kt)$ and $\cosh(kt)$, with $k\in \mathbb{N}$, as polynomials expressions of $\sin(t),\cos(t),\sinh(t)$ and $\cosh(t)$, via Chebyshev polynomials.
\State Compute $\cL_2(\ot)$ (see (\ref{eq:linearmap2})) and  \Return  $\cT(\cL_2(\cL_1(\ot)))$.		
\end{algorithmic}
\end{algorithm}

The next two algorithms focus on the conversion from trigonometric to rational and vice-versa.

\begin{algorithm}[H]
	\caption{{\bf [FromTrigToRat]} Obtains a rational parametrization from a hybrid trigonometric parametrization.} \vspace*{2mm}
	\label{alg:1}
	\begin{algorithmic}[1]

		\Require{A hybrid trigonometric parametrization $\mathcal{T}(\overline{t})$ of an algebraic variety $\vV$.}
		\Ensure{A rational parametrization $\mathcal{G}(\overline{t})$ of $\vV$.}
        \State \textbf{If} $\cT(\ot)$ is not pure, apply Algorithm \ref{alg:0} \textbf{end if}.
	     \State Determine the rational parametrization $\mathcal{G}(\overline{t})= \mathcal{F}(\mathcal{M}(\ot)) $ (see  (\ref{eq-cF}) for the definition of $\mathcal{F}$ and (\ref{eq-param-toro}) for the definition of $\mathcal{M}$).		
\State \Return $\mathcal{G}(\overline{t})$.
\end{algorithmic}
\end{algorithm}

\begin{example}\label{ex1}
We consider the $(1,1,0)$-hybrid trigonometric parametrization
\[ \mathcal{T}(t_1,t_2)=\left( \cos \left( t_{{1}}\right)^{2} \sin \left( t_{{1}}
 \right) ,{\frac {\sin \left( t_{{1}} \right) }{\sinh \left( t_{{2}}
 \right) }}, \sin \left( t_{{1}} \right)^{3} \right)
 \]
 of $\vV$.  We apply Algorithm \ref{alg:1}. In the first step,
we observe that $\cT(\ot)$ is not pure, since $\cosh(t_2)$ does not appear in the tuple. So,   we replace $\mathcal{T}(t_1,t_2)$ by $\mathcal{T}(t_1,2t_2)$. The new parametrization is
 \[
 \mathcal{T}(t_1,t_2)=
 \left( \cos \left( t_{{1}} \right)^{2}\sin \left( t_{{1}}
 \right), {\frac {\sin \left( t_{{1}} \right) }{2 \sinh\left( t_{{
2}} \right) \cosh \left( t_{{2}} \right) }}, \sin \left( t_{{1}
} \right)^{3}  \right).
 \]
 So, we have that
 \[ \mathcal{F}(y_{11},y_{12},y_{21},y_{22})=
 \left(
 y_{{11}}^{2}y_{{12}},\frac{y_{{12}}}{{2 y_{{21}}y_{{22}}}}
,y_{{12}}^{3}
 \right)
 \]
 Finally, we get the rational parametrization of $\vV$ is
 \[
 \mathcal{G}(t_1,t_2)=\left( 4\,{\frac {t_{{1}}^{2} \left( t_{1}^{2}-1 \right) }{ \left(
 t_{1}^{2}+1 \right)^{3}}},2\,{\frac { \left( t_{1}^{2}-1
 \right) t_{2}^{2}}{ \left( t_{1}^{2}+1 \right)  \left( t_{2}^{4}-1 \right) }},{\frac { \left( t_{1}^{2}-1 \right) ^{3}}{
 \left( t_{1}^{2}+1 \right) ^{3}}}
 \right).
 \]
 In Example \ref{ex3}, we see that  $\vV$ is the surface $x_{1}^3+3x_{1}^2x_{3}+3x_{1}x_{3}^2+x_{3}^3-x_3=0$.
\end{example}

\begin{algorithm}[H]
	\caption{{\bf [FromRatToTrig]} Obtains a hybrid trigonometric parametrization from a rational  parametrization.} \vspace*{2mm}
	\label{alg:2}
	\begin{algorithmic}[1]

		\Require{A rational parametrization $\mathcal{P}(\overline{t})$ of an $m$--dimensional unirational variety $\vV$ as well as a non-negative integer triple $\om=(m_1,m_2,m_3)$  such that $m_1+m_2+m_3=m$.}
		\Ensure{A $\om$--hybrid trigonometric parametrization $\mathcal{Q}(\overline{t})$ of $\vV$.}
		\State Compute $\mathcal{Q}(\ot)=\cP(\mathcal{M}^{-1}(\Psi(\ot)))$ (see Diagram (\ref{eq:diagram 1})): for $\mathcal{M}^{-1}$ and $\Psi$ see (\ref{eq-inversa-toro}) and (\ref{eq-psi}), respectively.
\State \Return $\mathcal{Q}(\overline{t})$.
\end{algorithmic}
\end{algorithm}

\begin{example}\label{ex2}
We apply Algorithm \ref{alg:2} to the unit circle parametrized by
\[ \cP(t)=\left(\dfrac{2t}{t^2+1},\dfrac{t^2-1}{t^2+1}\right) \]
taking $\om=(1,0,0)$. In this case,
\[ \mathcal{M}^{-1}=\frac{x_1}{1-x_2}, \,\,\,\text{and}\,\,\,\Psi(t)=(\cos(t),\sin(t)) \]
and the algorithm returns the expected parametrization $\cQ(t)=(\cos(t),\sin(t))$.  Alternatively, we may consider the same
parametrization $\cP(t)$ and $\om=(0,1,0)$. In this case, we
\[ \mathcal{M}^{-1}=\frac{1}{x_1-x_2}, \,\,\,\text{and}\,\,\,\Psi(t)=(\cosh(t),\sinh(t)) \]
and the algorithm returns the hyperbolic trigonometric parametrization
\[\cQ(t)=\left(\dfrac{1}{\cosh(t)}, \dfrac{\sinh(t)}{\cosh(t)}\right).\]
\end{example}

Now we deal with the problem of implicitizing  a hybrid trigonometric parametrization. Since we already have an algorithm to compute a rational parametrization of the variety, namely Algorithm \ref{alg:1}, one can simply
apply the existing implicitization techniques to that rational parametrization.  Alternatively, one may use the implicit equations of the circles and hyperbolas involved in the input parametrization. More precisely, one has the
following algorithm.

\begin{algorithm}[H]
	\caption{{\bf [FromTrigParamToImpl]} Obtains the implicit equations from a hybrid trigonometric parametrization.} \vspace*{2mm}
	\label{alg:3}
	\begin{algorithmic}[1]

		\Require{A hybrid trigonometric parametrization $\mathcal{T}(\overline{t})$ of the hybrid trigonometric variety $\vV$.}
		\Ensure{A set of polynomials defining $\vV$.}
\newline
        \mbox{\textbf{Option 1}}		
\State Apply Algorithm \ref{alg:1} to get a rational parametrization $\mathcal{G}(\ot)$ of $\vV$.
\State Implicizite $\mathcal{G}(\ot)$ and \Return the output.
\newline
        \mbox{\textbf{Option 2}}
       \State \textbf{If} $\cT(\ot)$ is not pure, apply Algorithm \ref{alg:0} \textbf{end if}.
		\State Determine  $\mathcal{F}$ (see (\ref{eq-cF})); say $\cF=(f_1(\oy)/g_1(\oy),\ldots,f_n(\oy),g_n(\oy))$.
        \State Eliminate $\{W,\oy\}$ from $\{g_i(\oy) x_i -f_i(\oy)\}_{i=1,\ldots,n}\cup \{W\,\mathrm{lcm}(g_1,\ldots,g_n)-1\} \cup \mathcal{H}(\oy)$, where $\mathcal{H}$ is the set of generators of $\toro$ (see (\ref{eq:impl-toro})).
\State \Return the result of the previous step.
\end{algorithmic}
\end{algorithm}

\begin{example}\label{ex3}
Let $\cT(\ot)$ be the parametrization in Example \ref{ex1}, and $\mathcal{G}$ be the rational parametrization generated by Algorithm \ref{alg:1} (see Example \ref{ex1}). Implicitizing $\mathcal{G}$, that is using Option 1 in Algorithm \ref{alg:3}, one gets that $x_{1}^3+3x_{1}^2x_{3}+3x_{1}x_{3}^2+x_{3}^3-x_3$ is the implicit equation of $\vV$. Alternatively, one may use Option 2 in Algorithm \ref{alg:3}. Proceeding as in Example \ref{ex1}, we get
\[ \mathcal{F}(y_{11},y_{12},y_{21},y_{22})=
 \left(
 y_{{11}}^{2}y_{{12}},\frac{y_{{12}}}{{2 y_{{21}}y_{{22}}}}
,y_{{12}}^{3}
 \right).
 \]
 Moreover, the implicit equations of $\mathcal{H}$ are $\mathcal{H}=\{y_{11}^{2}+y_{12}^{2}-1,y_{21}^{2}-y_{22}^{2}-1\}$. Let $\mathrm{J}$ be the ideal generated by
 $\{x_1 -y_{{11}}^{2}y_{{12}},x_2 y_{12}-2 y_{21}y_{22}, x_3-y_{12}^{3} , y_{11}^{2}+y_{12}^{2}-1,y_{21}^{2}-y_{22}^{2}-1,Wy_{21}y_{22}-1\}$. Using a suitable Gr\"obner basis we get that $\mathrm{J}\cap \C[\ox]$ is generated by $\{ x_{1}^3+3x_{1}^2x_{3}+3x_{1}x_{3}^2+x_{3}^3-x_3\}$.
\end{example}

\begin{example}\label{ex33}
We consider the trigonometric variety $\vV$ in $\R^4$ with associated $\om=(2,0,0)$--parametrization
$$\mathcal{T}(t_1,t_2)=\left(\frac{1}{\cos(t_1)},\frac{\cos(t_2)}{\sin(t_1)},\frac{1}{\sin(t_1)},\frac{\cos(t_2)}{\sin(t_2)}\right).$$
Applying  Algorithm~\ref{alg:1} we obtain the rational parametrization
\[\mathcal{G}(\ot)=
\left( \frac{t_{1}^{2}+1}{2 t_{1}},
\frac{2 t_{2} \left(t_{1}^{2}+1 \right) }{ \left( t_{2}^{2}+1 \right)  \left( t_{1}^{2}-1 \right)},
\frac{t_{1}^{2}+1}{t_{1}^{2}-1},\frac{2t_{2}}{t_{2}^{2}-1}\right).\]
Applying Algorithm \ref{alg:3}, we get that $\vV$ is the surface of $\R^4$ defined by
\[
\{x_{1}^{2}x_{3}^{2}-x_{1}^{2}-x_{3}^{2}=0, x_{2}^{2}x_{4}^{2}-x_{3}^{2}x_{4}^2+x_{2}^2=0\}.
\]
\end{example}

\section{Motivating Examples of Applicability}\label{sec-aplicactions}

In this section, by means of some examples we illustrate some potential applications that motivate the use of the theory developed in the previous sections.

\subsection{Plotting}\label{subsec-plotting}

It is well known that plotting geometric objects using a parametric representation is more suitable.
In the next example we show that using an $(m_1,0,0)$--trigonometric parametrization can have advantages over a unirational
parametrization; the key idea is that in the first case the behavior of the parametrization is controlled when the parameters take values in a bounded set.

\begin{example}\label{ex:plot}
We consider the $(1,0,0)$-trigonometric parametrization
\[ \cT(t)=\left( \dfrac{(1+\cos(5t))\sin(t)}{1-\cos(t)},(1+\cos(5t))(\cos(t))\right). \]
Applying  Algorithm~\ref{alg:1} we obtain the rational parametrization
\[\mathcal{G}(t)=
\left( {\frac { \left( t+1 \right) ^{3} \left( {t}^{4}+4 {t}^{3}-14 {t}^{2
}+4 t+1 \right) ^{2}}{ \left( t-1 \right)  \left( {t}^{2}+1 \right) ^
{5}}},2 {\frac {t \left( t+1 \right) ^{2} \left( {t}^{4}+4 {t}^{3}-
14 {t}^{2}+4 t+1 \right) ^{2}}{ \left( {t}^{2}+1 \right) ^{6}}}\right).\]
Applying Algorithm \ref{alg:3}, we get that $\vV$ is the curve of $\R^2$ defined by
\[
\begin{array}{r} x_{1}^{12}x_{2}-2 x_{1}^{12}+42 x_{1}^{10}x_{2}^{2}-344 x_{1}^{8}x_{2}^{4}+32 x_{1}^
{6}x_{2}^{6}+23 x_{1}^{10}y-623 x_{1}^{8}x_{2}^{3}+3304 x_{1}^{6}x_{2}^{5}+
11824 x_{1}^{4}x_{2}^{7}\\ +3712 x_{1}^{2}x_{2}^{9}+256 x_{2}^{11}+2 x_{1}^{10}-
80 x_{1}^{8}x_{2}^{2}+1120 x_{1}^{6}x_{2}^{4}-6400 x_{1}^{4}x_{2}^{6}+12800 {x_{1}
}^{2}x_{2}^{8}
=0.
\end{array}
\]
\begin{figure}
	\centering
$\begin{array}{cc}
		\mbox{\includegraphics[width=6cm]{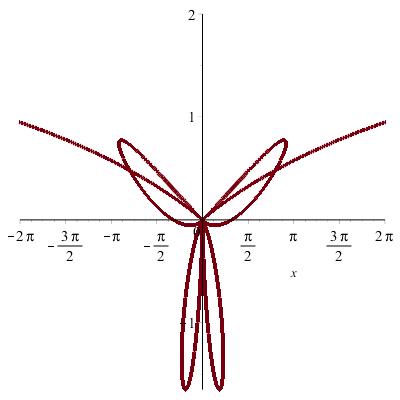}}
&
\mbox{\includegraphics[width=6cm]{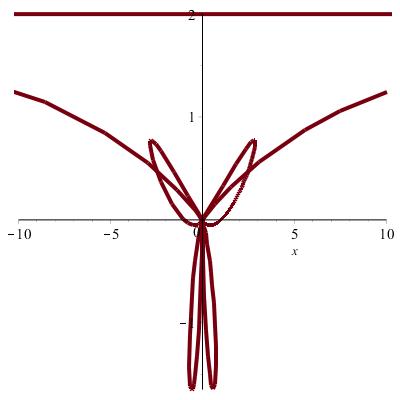}}
\end{array}$
		\caption{Example \ref{ex:plot}: plot generated by Maple using $\cT(t)$ (left) and using $\mathcal{G}(t)$ (right) }
\label{fig-1}
\end{figure}

In Fig. \ref{fig-1} we plot the real curve $\vV$. Both plots have been generated using Maple. The plot on the left was obtained using the trigonometric parametrization $\cT(t)$ with $t\in [-2\pi,2\pi]$. The plot on the right used the rational parametrization $\mathcal{G}(t)$ with $t\in [-10,10]$. One may observe that Maple, using the rational parametrization introduces wrongly the line $y=2$ which corresponds to the asymptote of the curve when $t$ tends to $-1$.
\end{example}

\subsection{Epicycloid and hypocycloid surfaces}\label{subsec-epi-hypo}

In this subsection,
we show how the classical epicycloid and hypocycloid constructions for circles can be generalized to the case of spheres, generating naturally examples of trigonometric parametrizations.  An epicycloid is a plane curve drawn by a fixed point in a circle rolling without slipping around a second fixed circle. This is a very classical curve which has been widely studied (e.g. see~\cite{lawrence}).

A natural generalization of such construction to a higher number of variables describes the trail of a fixed point in a sphere within the affine space of dimension 3, rolling around a second fixed sphere (see Fig. \ref{fig-2}). This phenomenon can be generally described in terms of a parametrization of a surface. Assume the fixed sphere is centered at the origin, with radius $R>0$, and the moving one has radius $0<r\le R$. In the case of the rolling sphere being of larger radius, an analogous construction can be made.

\begin{figure}
	\centering
		\includegraphics[width=6cm]{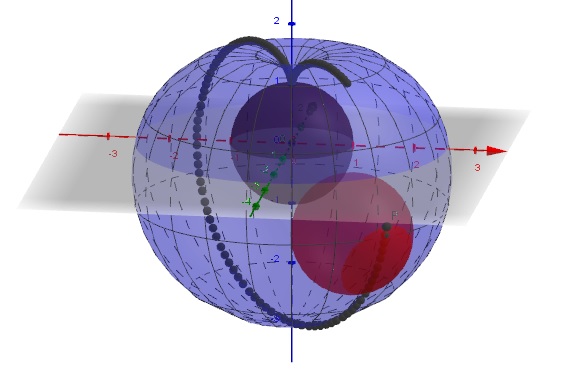}
		\caption{Construction of a 2-dimensional epicycloid $R=r=1$}\label{fig-2}
\end{figure}

After the application of an affinity on the space, one can describe the generalized epicycloid in terms of the following parametrization:

\begin{equation}\label{eq-epi}
\begin{array}{cccl}
\mathcal{T}:&[0,\pi]\times[0,2\pi) &\longrightarrow & \mathbb{R}^3\\
      & (t_1,t_2) & \longmapsto  & \left((R+r)\sin(t_1)\cos(t_2)-r\sin\left((1+\frac{R}{r})t_1\right)\cos(t_2)\right.,\\
			& \phantom{m} & \phantom{m} & (R+r)\sin(t_1)\sin(t_2)-r\sin\left((1+\frac{R}{r})t_1\right)\sin(t_2),\\
		& \phantom{m} & \phantom{m} &	\left. (R+r)\cos(t_1)-r\cos\left((1+\frac{R}{r})t_2\right)\right).
\end{array}
\end{equation}

 In the sequel, let us assume that $R/r$ is a rational number. In this situation, $\mathcal{T}$ is a $(2,0,0)$--trigonometric parametrization (see Definition \ref{def:param-trig}) and the epicycloid the surface that it generates. Moreover, one can rewrite it in pure form  (see Lemma~\ref{lemma-standard}) as well as Algorithm \ref{alg:0}.

 Let us illustrate the construction with  a particular example.

\begin{example}\label{ex:epi}
Consider the case of $R=5$ and $r=1$ (see Fig. \ref{fig-3}). Then, $\mathcal{T}$ can be writen in the form
$$\mathcal{T}(\ot)= \left(\mathcal{T}_1(t_1,t_2),\mathcal{T}_2(t_1,t_2),\mathcal{T}_3(t_1,t_2)\right),$$
where
\begin{equation}\label{eq:epi-paramT}
\begin{array}{lll}
\mathcal{T}_1(\ot)&=&6\sin(t_1)\cos(t_2)-32\cos(t_2)\sin(t_1)\cos(t_1)^5+32\cos(t_2)\sin(t_1)\cos(t_1)^3\\
&&-6\cos(t_2)\sin(t_1)\cos(t_1), \\
\mathcal{T}_2(\ot)&=&6\sin(t_1)\sin(t_2)-32\sin(t_2)\sin(t_1)\cos(t_1)^5+32\sin(t_2)\sin(t_1)\cos(t_1)^3\\
&& -6\sin(t_2)\sin(t_1)\cos(t_1), \\
\mathcal{T}_3(\ot)&=& 6\cos(t_1)-32\cos(t_1)^6+48\cos(t_1)^4-18\cos(t_1)^2+1.
\end{array}
\end{equation}
\begin{figure}
	\centering
		\includegraphics[width=5cm]{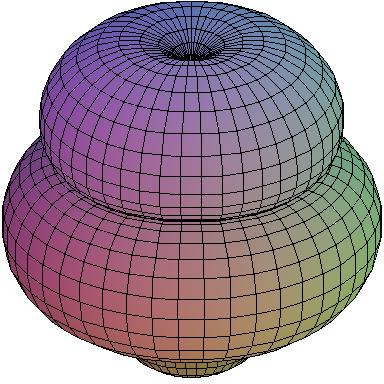}
		\caption{Generalized epicycloid in Example \ref{ex:epi}}\label{fig-3}
\end{figure}

Applying Algorithm \ref{alg:3} we get that the implicit equation of the epicycloid is given by the following polynomial

\noindent
$-52521875-42 x_{2}^{10}+90 x_{3}^4 x_{2}^4 x_{1}^4-252 x_{2}^2 x_{1}^4 x_{3}^2-252 x_{2}^2 x_{1}^2 x_{3}^4-252 x_{2}^4 x_{1}^2 x_{3}^2-840 x_{2}^2 x_{1}^6 x_{3}^2
+60 x_{3}^2 x_{2}^6 x_{1}^4+30 x_{3}^8 x_{2}^2 x_{1}^2-2436 x_{1}^2 x_{3}^4+60 x_{3}^2 x_{2}^4 x_{1}^6+x_{2}^{12}-812 x_{3}^6-4872 x_{2}^2 x_{1}^2 x_{3}^2-2436 x_{3}^4 x_{2}^2-210 x_{3}^8 x_{2}^2
-1286250 x_{1}^2-64050 x_{2}^2 x_{3}^2-1286250 x_{2}^2-42 x_{1}^{10}+x_{1}^{12}+60 x_{3}^6 x_{2}^4 x_{1}^2+20 x_{3}^6 x_{2}^6-64050 x_{1}^2 x_{3}^2+20 x_{2}^6 x_{1}^6
-84 x_{2}^6 x_{3}^2+466560 x_{3} x_{1}^4-1260 x_{2}^4 x_{1}^4 x_{3}^2-1260 x_{2}^4 x_{1}^2 x_{3}^4-1260 x_{2}^2 x_{1}^4 x_{3}^4-210 x_{1}^8 x_{3}^2+6 x_{3}^2 x_{1}^{10}+6 x_{2}^{10} x_{3}^2
+15 x_{3}^4 x_{2}^8+x_{3}^{12}+933120 x_{3} x_{2}^2 x_{1}^2-84 x_{2}^6 x_{1}^2-21 x_{1}^8-933120 x_{2}^2 x_{3}^3+15 x_{2}^4 x_{1}^8-210 x_{2}^8 x_{1}^2+466560 x_{3} x_{2}^4-210 x_{2}^2 x_{1}^8
-2436 x_{1}^4 x_{3}^2+60 x_{3}^4 x_{2}^6 x_{1}^2-84 x_{2}^2 x_{1}^6-420 x_{3}^4 x_{2}^6-420 x_{1}^4 x_{3}^6-420 x_{1}^6 x_{3}^4-1286250 x_{3}^2-126 x_{3}^4 x_{2}^4-933120 x_{1}^2 x_{3}^3
+93312 x_{3}^5-840 x_{2}^2 x_{1}^2 x_{3}^6+60 x_{3}^4 x_{2}^2 x_{1}^6+30 x_{3}^2 x_{2}^2 x_{1}^8+30 x_{3}^2 x_{2}^8 x_{1}^2-210 x_{2}^8 x_{3}^2-420 x_{2}^6 x_{1}^4-84 x_{1}^2 x_{3}^6
+15 x_{3}^8 x_{1}^4+20 x_{3}^6 x_{1}^6+15 x_{3}^4 x_{1}^8-2436 x_{2}^2 x_{1}^4+6 x_{3}^{10} x_{2}^2-420 x_{2}^4 x_{3}^6-84 x_{1}^6 x_{3}^2-32025 x_{2}^4-126 x_{1}^4 x_{3}^4+15 x_{2}^8 x_{1}^4
-2436 x_{2}^4 x_{3}^2-32025 x_{3}^4+60 x_{3}^6 x_{2}^2 x_{1}^4-21 x_{2}^8-812 x_{2}^6-420 x_{2}^4 x_{1}^6+15 x_{3}^8 x_{2}^4-21 x_{3}^8-840 x_{2}^6 x_{1}^2 x_{3}^2-210 x_{1}^2 x_{3}^8
-32025 x_{1}^4-2436 x_{2}^4 x_{1}^2+6 x_{2}^2 x_{1}^{10}-812 x_{1}^6-42 x_{3}^{10}+6 x_{2}^{10} x_{1}^2+6 x_{3}^{10} x_{1}^2-126 x_{2}^4 x_{1}^4-64050 x_{2}^2 x_{1}^2-84 x_{3}^6 x_{2}^2.
$

\noindent Applying Algorithm \ref{alg:1} we get the following rational parametrization of the epicycloid
\begin{equation}\label{eq:epi-paramR}
\begin{array}{lll}
\mathcal{G}(\ot)&=& \left(4 \dfrac { \left( t_{1}^{2}-1 \right) t_{2} g_1(\ot)}{ \left( t_{2}^{2}+1 \right)  \left( t_{1}^{2}+1
\right) ^{6}},
2 \dfrac{ \left( t_{1}^{2}-1 \right)  \left( t_{2}^{2}-1 \right) g_1(\ot)}{ \left( t_{2}^{2}+1 \right)  \left( t_{1}^{2}+1 \right) ^{6}},
\dfrac{g_2(\ot)}{\left( t_{1}^{2}+1 \right)^{6}}\right),
\end{array}
\end{equation}
where $g_1(\ot)=3 t_{1}^{10}-6 t_{1}^{9}+
15 t_{1}^{8}+104 t_{1}^{7}+30 t_{1}^{6}-292 t_{1}^{5}+30 t_{1}^{4}+104 t_{1}
^{3}+15 t_{1}^{2}-6 t_{1}+3$ and $g_2(\ot)=t_{1}^{12}+12 t_{1}^{11}-66 t_{1}^{10}+60 t_{1}^{9}+495 t_{1}^{8}+
120 t_{1}^{7}-924 t_{1}^{6}+120 t_{1}^{5}+495 t_{1}^{4}+60 t_{1}^{3}-66
t_{1}^{2}+12 t_{1}+1$.
\end{example}

We can adapt the previous reasoning to the case of hypocycloids. The construction of a classical hypocycloid is analogous to that of a cycloid. Here, the moving disc is rolling inside the fixed one (see Fig \ref{fig-hypo}). We consider the generalization in which a sphere of radius $r>0$ is rolling inside a fixed one of radius $r<R$.

\begin{figure}
	\centering
		\includegraphics[width=6cm]{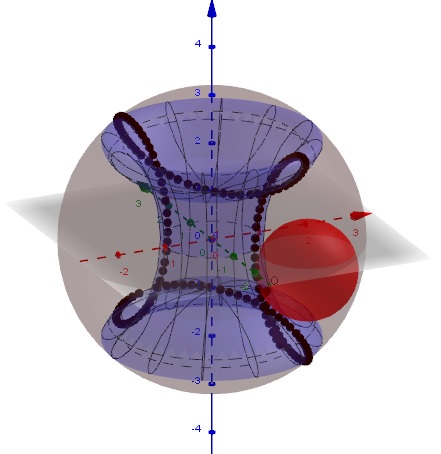}
		\caption{Construction of hypocycloid $R=3$,$r=1$}\label{fig-hypo}
\end{figure}

We assume again that $R/r$ is a natural number. After the application of an affinity on the space, one can describe the generalized hypocycloid in terms of the following $(2,0,0)$-trigonometric parametrization:
\begin{equation}\label{eq-epi2}
\begin{array}{cccl}
\mathcal{T}:&[0,\pi]\times[0,2\pi) &\longrightarrow & \mathbb{R}^3\\
      & (t_1,t_2) & \longmapsto  & \left((R-r)\sin(t_1)\cos(t_2)-r\sin\left(\frac{R}{r}t_1\right)\cos(t_2)\right.,\\
			& \phantom{m} & \phantom{m} & (R-r)\sin(t_1)\sin(t_2)-r\sin\left(\frac{R}{r}t_1\right)\sin(t_2),\\
		& \phantom{m} & \phantom{m} &	\left. (R-r)\cos(t_1)-r\cos\left(\frac{R}{r}t_2\right)\right).
\end{array}
\end{equation}

\begin{example}\label{ex-hypo}
Consider the case of $R=7$ and $r=1$ (see Fig \ref{fig-hypoEx}). Then, $\mathcal{T}$ can be written in the form
$$\mathcal{T}(t_1,t_2)= \left(\mathcal{T}_1(t_1,t_2),\mathcal{T}_2(t_1,t_2),\mathcal{T}_3(t_1,t_2)\right),$$
where
\[
\begin{array}{lll}
\mathcal{T}_1(\ot)&=&5\sin(t_1)\cos(t_2)+64\cos(t_2)\sin(t_1)\cos(t_1)^6-80\cos(t_2)\sin(t_1)\cos(t_1)^4\\
&&+24\cos(t_2)\sin(t_1)\cos(t_1)^2, \\
\mathcal{T}_2(\ot)&=&5\sin(t_1)\sin(t_2)+64\sin(t_2)\sin(t_1)\cos(t_1)^6-80\sin(t_2)\sin(t_1)\cos(t_1)^4\\
&&+24\sin(t_2)\sin(t_1)\cos(t_1)^2,\\
\mathcal{T}_3(\ot)&= &13\cos(t_1)-64\cos(t_1)^7+112\cos(t_1)^5-56\cos(t_1)^3.
\end{array}
\]

\begin{figure}
	\centering
		\includegraphics[width=5cm]{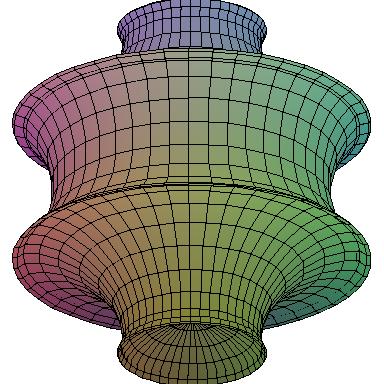}
		\caption{Example of generalized hypocycloid in Example \ref{ex-hypo}}\label{fig-hypoEx}
\end{figure}

Applying Algorithm \ref{alg:3}, we get the following polynomial that defines the hypocyclod.

\noindent
$-2251875390625+4245 x_{3}^8 x_{2}^2+35 x_{2}^6 x_{1}^8+x_{2}^{14}+584965 x_{2}^8+x_{3}^{14}+x_{1}^{14}+174 x_{2}^{10} x_{1}^2
+105 x_{2}^4 x_{1}^8 x_{3}^2+140 x_{2}^2 x_{1}^6 x_{3}^6+16980 x_{3}^6 x_{2}^2 x_{1}^2+105 x_{2}^2 x_{1}^4 x_{3}^8+105 x_{2}^2 x_{1}^8 x_{3}^4
+42 x_{2}^2 x_{1}^{10} x_{3}^2+870 x_{2}^8 x_{1}^2 x_{3}^2+870 x_{2}^2 x_{1}^8 x_{3}^2+870 x_{2}^2 x_{3}^8 x_{1}^2+1740 x_{2}^2 x_{3}^6 x_{1}^4+1740 x_{2}^4 x_{1}^6 x_{3}^2
+210 x_{2}^4 x_{3}^6 x_{1}^4+16980 x_{2}^2 x_{1}^6 x_{3}^2+25470 x_{3}^4 x_{2}^4 x_{1}^2+42 x_{2}^2 x_{1}^2 x_{3}^{10}+25470 x_{3}^4 x_{2}^2 x_{1}^4
+7 x_{2}^{12} x_{1}^2+8490 x_{3}^4 x_{2}^6+8490 x_{3}^6 x_{1}^4-46728132 x_{2}^4 x_{1}^2 x_{3}^2+78683196 x_{2}^2 x_{3}^4 x_{1}^2+4490850 x_{2}^2 x_{1}^2 x_{3}^2
-46728132 x_{2}^2 x_{3}^2 x_{1}^4+42 x_{2}^{10} x_{1}^2 x_{3}^2+25470 x_{2}^4 x_{3}^2 x_{1}^4+140 x_{2}^6 x_{1}^6 x_{3}^2+16980 x_{2}^6 x_{1}^2 x_{3}^2
+105 x_{2}^4 x_{3}^8 x_{1}^2+210 x_{2}^4 x_{3}^4 x_{1}^6+210 x_{2}^6 x_{3}^4 x_{1}^4+105 x_{2}^8 x_{3}^4 x_{1}^2+140 x_{2}^6 x_{3}^6 x_{1}^2+105 x_{2}^8 x_{3}^2 x_{1}^4
+1740 x_{2}^2 x_{3}^4 x_{1}^6+2610 x_{2}^4 x_{3}^4 x_{1}^4+1740 x_{2}^6 x_{3}^4 x_{1}^2+1740 x_{2}^4 x_{3}^6 x_{1}^2+1740 x_{2}^6 x_{3}^2 x_{1}^4-15576044 x_{2}^6 x_{3}^2
+2339860 x_{2}^2 x_{1}^6+2339860 x_{2}^6 x_{1}^2+4245 x_{3}^8 x_{1}^2+8490 x_{3}^6 x_{2}^4+748475 x_{2}^6+29 x_{1}^{12}-15576044 x_{1}^6 x_{3}^2
+39341598 x_{1}^4 x_{3}^4+4245 x_{1}^8 x_{3}^2+3509790 x_{2}^4 x_{1}^4+849 x_{2}^{10}+35 x_{2}^6 x_{3}^8+2245425 x_{3}^2 x_{1}^4+435 x_{2}^8 x_{1}^4
+435 x_{2}^8 x_{3}^4+174 x_{2}^2 x_{1}^{10}+7 x_{2}^{12} x_{3}^2+8490 x_{2}^6 x_{1}^4+8490 x_{2}^4 x_{1}^6+849 x_{3}^{10}-15576044 x_{1}^2 x_{3}^6+45018750 x_{2}^2 x_{1}^2
+45018750 x_{1}^2 x_{3}^2+4245 x_{2}^8 x_{3}^2+7 x_{2}^2 x_{3}^{12}+35 x_{1}^8 x_{3}^6+7 x_{3}^2 x_{1}^{12}+7 x_{1}^2 x_{3}^{12}+35 x_{1}^6 x_{3}^8+435 x_{1}^8 x_{3}^4
+174 x_{1}^2 x_{3}^{10}+174 x_{1}^{10} x_{3}^2+435 x_{1}^4 x_{3}^8+580 x_{1}^6 x_{3}^6+2245425 x_{2}^4 x_{1}^2+435 x_{2}^4 x_{1}^8+35 x_{2}^8 x_{3}^6+2245425 x_{2}^4 x_{3}^2
+174 x_{2}^2 x_{3}^{10}+45018750 x_{2}^2 x_{3}^2+35 x_{2}^8 x_{1}^6+21 x_{2}^{10} x_{1}^4-15576044 x_{2}^2 x_{3}^6+21 x_{1}^{10} x_{3}^4+21 x_{1}^4 x_{3}^{10}
+2245425 x_{3}^4 x_{1}^2+4245 x_{2}^8 x_{1}^2+174 x_{2}^{10} x_{3}^2+2245425 x_{2}^2 x_{1}^4+748475 x_{3}^6+748475 x_{1}^6+7 x_{2}^2 x_{1}^{12}+39341598 x_{2}^4 x_{3}^4
+4245 x_{2}^2 x_{1}^8+849 x_{1}^{10}+22509375 x_{2}^4+22509375 x_{1}^4+21 x_{2}^{10} x_{3}^4+580 x_{2}^6 x_{1}^6+22509375 x_{3}^4+21 x_{2}^4 x_{1}^{10}
+79045421875 x_{3}^2+584965 x_{1}^8+584965 x_{3}^8+2245425 x_{2}^2 x_{3}^4+21 x_{2}^4 x_{3}^{10}+435 x_{2}^4 x_{3}^8+580 x_{2}^6 x_{3}^6+29 x_{2}^{12}
+79045421875 x_{1}^2+79045421875 x_{2}^2+8490 x_{3}^4 x_{1}^6+29 x_{3}^{12}.$
\end{example}

\subsection{Computing Intersections}\label{subsec-intersection}

Let us say that we want to compute the intersection of two algebraic surfaces. Usually one takes, if possible, a parametrization of one of the surfaces, and substitute it in the implicit equation of the other. This provides an equation that encodes the parameter values to be substituted in the parametrization to achieve the intersection set. In  the following example we see that if we are given a trigonometric parametrization (for instance when dealing with an epicycloid) the task is more difficult than using a rational parametrization.

\begin{example}\label{ex:inter}
In this example, we consider the epicycloid of Example \ref{ex:epi}, let us call it $\vV_1$, and the sphere $\vV_2$ of equation $x_{1}^{2}+x_{2}^{2}+x_{3}^{2}=36$. We want to compute $\vV_1\cap \vV_2$. The construction of the generalized epicycloid with such sphere suggests a nonempty intersection, as it can be observed in Figure~\ref{figaux1}.

\begin{figure}
	\centering
		\includegraphics[width=5cm]{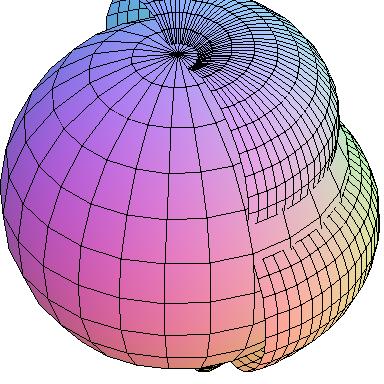}
		\caption{Detail on the intersection of an epicycloid with a sphere}\label{figaux1}
\end{figure}

Although the parametrization of $\vV_2$ is simple, the implicit equation of $\vV_1$ is huge. So, we try to use a parametrization of $\vV_1$ and the implicit equation of the sphere. If we use the trigonometric parametrization $\cT(\ot)$ (see (\ref{eq:epi-paramT})), we get the equation
\[ -192\cos(t_1)^5+240\cos(t_1)^3-60\cos(t_1)+1=0.\]
Maple provides only the solution $t_1=0.297473248$ that generates the circle parametrized as
\[ (0.7814521186\cos(t_2), 0.7814521186\sin(t_2), 5.94889339). \]
Taking into account Fig. \ref{figaux1}, this solution looks incomplete. However, if we use the rational parametrization $\mathcal{G}(\ot)$ given in (\ref{eq:epi-paramR}) we get the equation
\[  t_{1}^{10}-120t_{1}^{9}+5t_{1}^{8}+1440t_{1}^{7}+10t_{1}^{6}-3024t_{1}^{5}+10t_{1}^{4}+1440_{t}^{3}+5t_{1}^{2}-120 t_{1}+1=0,
\]
which roots are all real and can be approximated as
\[ \begin{array}{r} \{-2.992499717, -1.400811244, -0.7138720538, -0.3341687869, 0.008343202240, \\ 0.3157206162, 0.739367548, 1.352507292, 3.167357305, 119.8580558\}.
\end{array} \]
Substituting these two roots in $\mathcal{G}(\ot)$ we get the following five circles of intersection
\[
\begin{array}{l}
\left( 1.56290425{\dfrac{t_{2}}{t_{2}^{2}+ 1.0}},
 0.781452121{\dfrac {t_{2}^{2}- 1.0}{t_{2}^{2}+ 1.0}},
 5.948893399\right),\\ \\
 \left( 5.72892752{\dfrac {t_{2}}{t_{2}^{2}+ 1.0}},
 2.864463754{\dfrac {t_{2}^{2}- 1.0}{t_{2}^{2}+ 1.0}},-
 5.272081897\right),\\ \\
 \left( 8.25775970{\dfrac {t_{2}}{t_{2}^{2}+ 1.0}},
 4.128879847{\dfrac {t_{2}^{2}- 1.0}{t_{2}^{2}+ 1.0}},-
 4.353429820\right),\\ \\
 \left( 10.83250368{\dfrac {t_{2}}{t_{2}^{2}+ 1.0}},
 5.416251839{\dfrac {t_{2}^{2}- 1.0}{t_{2}^{2}+ 1.0}},
 2.581514281\right),\\ \\
 \left( 11.79843163{\dfrac {t_{2}}{t_{2}^{2}+ 1.0}},
 5.899215810{\dfrac {t_{2}^{2}- 1.0}{t_{2}^{2}+ 1.0}},
 1.095104026\right).
\end{array}
\]
In order to check that this last result is correct, we compute a Gr\"obner basis of the ideal of $\vV_1\cap \vV_2$ w.r.t. a lexicographic order to
get
\[ \{1492992 x_{3}^5-67184640x_{3}^3+604661760 x_{3}-576284939, x_{1}^2+x_{2}^2+x_{3}^2-36 \}.\]
The roots of the univariate polynomial in the basis are $$\{-5.272081883, -4.353429821, 1.095104026, 2.581514286, 5.948893392 \}$$ that are the level planes where the circles lie on. On the other hand substituting these 5 roots in the second polynomial we get the circles. In Fig. \ref{fig-int} we plot the five intersection circles.
\begin{figure}
	\centering
$\begin{array}{cc}
		\mbox{\includegraphics[width=6cm]{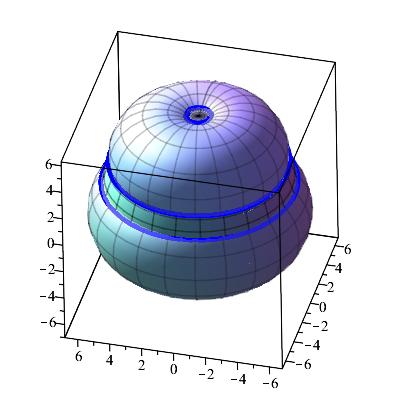}}
&
\mbox{\includegraphics[width=6cm]{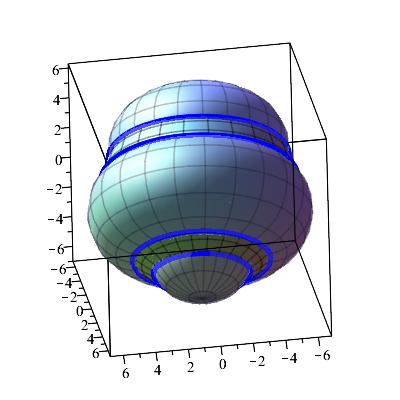}}
\end{array}$
		\caption{View of the intersection circles in Example \ref{ex:inter}. }
\label{fig-int}
\end{figure}
\end{example}

\subsection{Solving Differential Equations}\label{subsec-edo}

In the work~\cite{naher} the extended generalized Riccati Equation Mapping Method is applied to the (1+1)-Dimensional Modified KdV Equation
\begin{equation}\label{e1aux}
u_t-u^2u_x+\delta u_{xxx}=0,
\end{equation}
for some fixed $\delta>0$. More precisely, solutions of the generalized Riccati equation, namely
 \begin{equation}\label{eaux2}
G'=r+pG+qG^2,
\end{equation}
are used in order to provide solutions of (\ref{e1aux}). The authors arrive at different families of solutions, classified into soliton and soliton-like solutions (written in terms of rational functions of hyperbolic ones), and periodic solutions (written as rational functions of trigonometric ones), under different cases of the parameters involved.

In the following, we see how taking the particular solutions provided in \cite{naher} of (\ref{eaux2}) and using the ideas of this paper, we can generate all families of solutions in \cite{naher}. More precisely, using the notation in \cite{naher}, we take the solution of (\ref{eaux2})
\[ G_1(\eta)= -\dfrac{1}{2q} \left(p+\delta\tanh\left(\dfrac{\delta}{2}\eta \right) \right)\]
where $\delta=\sqrt {{p}^{2}-4\,qr}$. Therefore,
$$\cP(\eta)=(G_1(\eta),G_{1}^{\,\prime}(\eta))$$
 is a parametrization of the algebraic variety $\vV$ associated with  (\ref{eaux2}), namely the conic $y=r+px+qx^2$ (see \cite{fenggao2}). Since we do not whether $\delta$ is a rational number, $\cP(\eta)$ may not satisfy the conditions in Definition \ref{def:param-trig}. However the reparametrization
 $$\cT(\eta)=\cP\left(\frac{2}{\delta} \eta \right)
$$
 does, and it is a $(0,1,0)$--trigonometric parametrization of $\vV$.
 Algorithm \ref{alg:1} provides a rational parametrization of the variety, given by
 \[ \left( -{\frac {1}{2q} \left( p+{\frac {\delta\, \left( {\eta}^{2}-1
 \right) }{{\eta}^{2}+1}} \right) },-{\frac {{\delta}^{2}{\eta}^{2}}{q
 \left( {\eta}^{2}+1 \right) ^{2}}}\right), \]
that can be properly reparametrized as
\[ \mathcal{G}(\eta):=(g_1(\eta),g_2(\eta))=\left(-{\frac {\delta\,\eta+\eta\,p-\delta+p}{2q \left( \eta+1 \right)
}},-{\frac {{\delta}^{2}\eta}{q \left( \eta+1 \right) ^{2}}}\right). \]
The previous parametrization is no longer a solution of (\ref{eaux2}), so we search for a function $t\mapsto\phi(t)$ such that
$\mathcal{G}(\phi(t))$
provides a solution of (\ref{eaux2}). We ask the derivative of $g_1(\phi(t))$ with respect to $t$ to coincide with $g_2(\phi(t))$ to obtain a differential condition on $\phi$. Namely,
$$
-\delta\phi(t)+\phi(t)'=0,$$
with general solution
$$\phi(t)= C {\rm e}^{\delta t}.$$
So, we get the general solution
\[S(\eta,C)=-{\frac {\delta\,C\,{{\rm e}^{\delta\,\eta}}+C
\,{{\rm e}^{\delta\,\eta}}p-\delta+p}{2q \left(C\,{{\rm e}^{
\delta\,\eta}}+1 \right) }}\]
of the generalized Ricatti equation. Now, from $S(\eta,C)$ one may obtain the families of solutions in \cite{naher}. For instance, using the notation in \cite{naher}, $S(\eta,1)=G_1(\eta)$, $S(\eta,-1)=G_2(\eta)$, $S(\eta,\pm i)=G_3(\eta)$, $S(\eta,\mp 1)=G_4(\eta)$, etc.

 We observe that if one proceeds analogously replacing the rational parametrization by the hyperbolic parametrization $\cT(\eta)=(h_1(\eta),h_2(\eta))$, the procedure does not succeed. More precisely, we consider an unknown function $\psi(t)$ and search for all such functions which satisfy
$h_2(\psi(t))=\frac{d}{dt}(h_1(\psi(t))).$
We only get $\psi(t)=t+C$, for $C$ being an arbitrary constant.

\section{Conclusions} We have introduced a new type of parametrizations, namely those involving rationally circular and hyperbolic trigonometric functions and monomials being each of these three block depending of different sets of parameters. We have seen that the algebraic varieties defined by these new objects are precisely the real unirational varieties. In addition, we provide algorithms to deal with the computation of the generators of the variety, and to convert from trigonometric to unirational and viceversa. We have also illustrated by means of examples that having the option of parametrizing in these two different ways is an advantage for dealing with some applications; for some a rational parametrization is better, for others a trigonometric parametrization is more suitable.

\noindent \textbf{Acknowledgements.}  A. Lastra is supported by the Spanish Ministerio de Econom\'{\i}a, Industria y Competitividad under the Project MTM2016-77642-C2-1-P. J. R. Sendra  and J. Sendra are supported by the Spanish Ministerio de Econom\'{\i}a y Competitividad, and by the European Regional Development Fund (ERDF), under the Project MTM2014-54141-P.

The discussions with Mauro Beltrametti were very helpful. We are very grateful.     We also thank Daniel Pizarro for pointing out the example in the application in Subsection \ref{subsec-edo}.

\end{document}